\documentclass[12pt]{amsart}
\usepackage{amsmath,amssymb,amsfonts,amsthm,latexsym}
\usepackage{enumerate}
\usepackage{graphicx}

\numberwithin{equation}{section}
\newtheorem{thm}{Theorem}
\numberwithin{thm}{section}
\newtheorem{defn}[thm]{Definition}

\newtheorem{prop}[thm]{Proposition}

\newtheorem{question}{Question}

\begin{document}

\title[Boundary points with squeezing function one]{On Boundary 
points at which the squeezing function tends to one}

\author{Seungro Joo and Kang-Tae Kim}
\address{Department of Mathematics, POSTECH, 
Pohang 37673 The Republic of Korea}%
\email{beartan@postech.ac.kr (Joo)}
\email{kimkt@postech.ac.kr (Kim)}
\thanks{This research was supported in part by the Grant 2011-0030044 
(The SRC-GAIA) of the National Research Foundation of The Republic 
of Korea.}

\begin{abstract}
J.E. Forn{\ae}ss posed the question whether the boundary point of 
smoothly bounded pseudoconvex domain is strictly pseudoconvex, 
if the asymptotic limit of the squeezing function is 1. 
The purpose of this paper is to give an affirmative answer when the
domain is in \(\mathbb C^2\) with smooth boundary of finite type in
the sense of D'Angelo \cite{DAN}.
\end{abstract}

\maketitle

\section{Introduction}

Let \(\mathbb{B}^n (p; r) := \{z \in\mathbb{C}^n \colon \|z-p\| < r\}\).   
For a domain $\Omega$ in \(\mathbb{C}^n\) and 
$z_0 \in \Omega$, let
\(
\mathcal{F}_{\Omega}(z_0) := \{ f \colon \Omega \to \mathbb{B}^n (0;1) 
\mid f \textrm{ 1-1 holomorphic}, f(z_0) = 0 \} \).
Then the \textit{squeezing function} $s_{\Omega} \colon \Omega 
\to \mathbb R$ is defined to be
$$
s_{\Omega}(z) := \sup \{r \colon \mathbb{B}^n (0;r)
\subset f(\Omega), f \in \mathcal{F}_{\Omega}(z) \}.
$$
Note that $0<s_{\Omega}(z) \le 1$ for any \(z \in \Omega\).

This concept first appeared in \cite{LSY1, LSY2} in the context 
concerning the \textit{holomorphic homogeneous regular} manifolds.  
But the name squeezing function comes from \cite{DGZ}; 
this concept is closely related to the concept of 
\textit{bounded geometry} by Cheng and Yau \cite{CY}, as one sees
from \cite{Yeung}.  It is obvious from
its construction that the squeezing function is a biholomorphic invariant.
\medskip

Recent studies have shown:

\begin{thm}[\cite{KZ}, See also \cite{DFW, DGZ}]
If a bounded domain $\Omega$ in $\mathbb{C}^n$ has a boundary 
point, say \(p\), about which the boundary is strictly pseudoconvex, 
then $\displaystyle\lim_{\Omega \ni z \to p} s_{\Omega}(z) = 1$.
\end{thm}

J. E. Forn{\ae}ss asked recently whether its converse 
is true (cf.\ \cite{FW}, Sections 1 and 4), i.e., he posed the following question.  

\begin{question} \label{Fornaess}
If $\Omega$ is a bounded domain with smooth boundary, and if  
$\displaystyle \lim_{\Omega \ni z \to p \in \partial\Omega} 
s_{\Omega}(z) = 1$, then is the boundary of $\Omega$ strictly 
pseudoconvex at \(p\)?
\end{question}

In a recent article, A. Zimmer has shown that the answer is affirmative
if the bounded domain is also assumed to be convex \cite{Zimm}. 
\medskip

The main purpose of this article is to present the following result.

\begin{thm} \label{main}
Let $\Omega$ be a bounded domain in $\mathbb{C}^2$ with 
smooth pseudoconvex boundary. If $p$ is a boundary point of $\Omega$
of finite type, in the sense of D'Angelo \cite{DAN}, and if 
$\lim_{\Omega \ni z \to p} S_{\Omega}(z) = 1$, then \(\partial\Omega\) 
is strictly pseudoconvex at $p$.
\end{thm}

We remark that our proof is only for complex dimension 2. It is mainly 
due to the limitation of current knowledge concerning the convergence 
of the scaling methods (cf.\ \cite{GKK}).  If the domain is convex
and Kobayashi hyperbolic for instance, then there is no such restriction. 
Consequently in case the domain is bounded convex, as treated in 
\cite{Zimm}, our proof-arguments also answer Question \ref{Fornaess} 
affirmatively, which we shall put an explication of, in a remark at the 
end.

On the other hand, as shown recently in \cite{FW}, the answer to 
Question \ref{Fornaess} is negative if no other conditions on 
the bounded domain than the boundary being $\mathcal C^2$ 
are assumed.  Also the high dimensions can imply some unexpected
phenomenon \cite{FR}. Thus the hypothesis of the our theorem is 
in some sense reasonable.

\section{Construction of the scaling sequence}

Let $\Omega$ and the boundary point $p \in \partial \Omega$ be
as in the hypothesis of the theorem. 
\smallskip

Let \(v_p\) be the unit normal vector to \(\partial\Omega\) at $p$
pointing outward.  Then take a sequence $\{ p_j \} \subset \Omega$ 
such that \(p - p_j = t_j v_p\) for some \(t_j\) satisfying 
\begin{itemize}
\item[(1)] \( 0 < t_{j+1} < t_j \) for every \(j = 1,2,\ldots\) and 
\item[(2)] \(\lim_{j\to\infty} t_j = 0\). 
\end{itemize}
\smallskip

We also set $\delta_j := 2(1 - s_{\Omega}(p_j))$ for every $j$. 
Since $\lim_{j \to \infty} s_{\Omega}(p_j) = 1$ by assumption, there 
exists, for each \(j\), an injective holomorphic map 
$f_j \colon \Omega \to \mathbb{B}^2$ such that $f_j(p_j) = (0, 0)$ and 
$\mathbb{B}^2(0; 1 - \delta_j) \subset f_j(\Omega)$ for all $j$.
\medskip

The next step is to use the dilation sequence 
$\{ \alpha_j : \Omega \to \mathbb{C}^2 \}$ introduced in 
\cite{JooSR}, Section 2.2. (We remark that this is a mild but 
necessary modification of Pinchuk's stretching sequence \cite{BP, Pin}).
If the type of $p$ is $2k$ for some integer \(k\), then there is a 
neighborhood $U$ of $p$ and $\alpha_j \in Aut(\mathbb{C}^2)$ 
satisfying the following properties:
\begin{enumerate}
\item The map $\alpha_j$ is the composition (in order) of a 
translation, a unitary map, a triangular map and a dilation map.
\item $\alpha_j(p) = (0, 0)$ and $\alpha_j(p_j) = (-1, 0)$ for all $j$.
\item The local defining function $\rho_j$ of $\alpha_j (\Omega \cap U)$ 
at $(0, 0)$ is represented by
\begin{equation}
\begin{split}
\rho_j(w, z) = \textrm{Re}\,w &+ P_j(z, \bar{z}) + R_j(z, \bar{z}) \\
& \qquad + (\textrm{Im } w) Q_j \left( \textrm{Im } w, z, \bar{z} \right),
\end{split}
\end{equation}
where:
\begin{itemize} 
\item $P_j$ is a nonzero real-valued homogeneous subharmonic 
polynomial of degree $2k$ with no harmonic terms, 
\item $R_j$ and $Q_j$ 
are real-valued smooth functions satisfying the conditions on the 
vanishing:
order $\nu(R_j(z, \bar{z})) > 2k$ and\break $\nu \left(Q_j 
\left( \textrm{Im}w, z, \bar{z} \right) \right) \ge 1$.
\end{itemize}
\end{enumerate}

\noindent
Note that the convergences $P_j \to \widehat{P}, R_j \to 0$ and 
$Q_j \to 0$ are uniform on compact subsets of $\mathbb{C}^2$, while
$\widehat{P}$ is a nonzero real-valued subharmonic polynomial of 
degree $2k$. This is proved in detail in Lemma 2.4 of \cite{JooSR}.  
Consequently, $\rho_j$ converges uniformly to 
\(\widehat{\rho} := \textrm{Re}\,w + \widehat{P}(z, \bar{z})\) 
on compact subsets of $\mathbb{C}^2$.

\begin{figure}[!h]
\centering
\includegraphics[width=0.8\columnwidth]{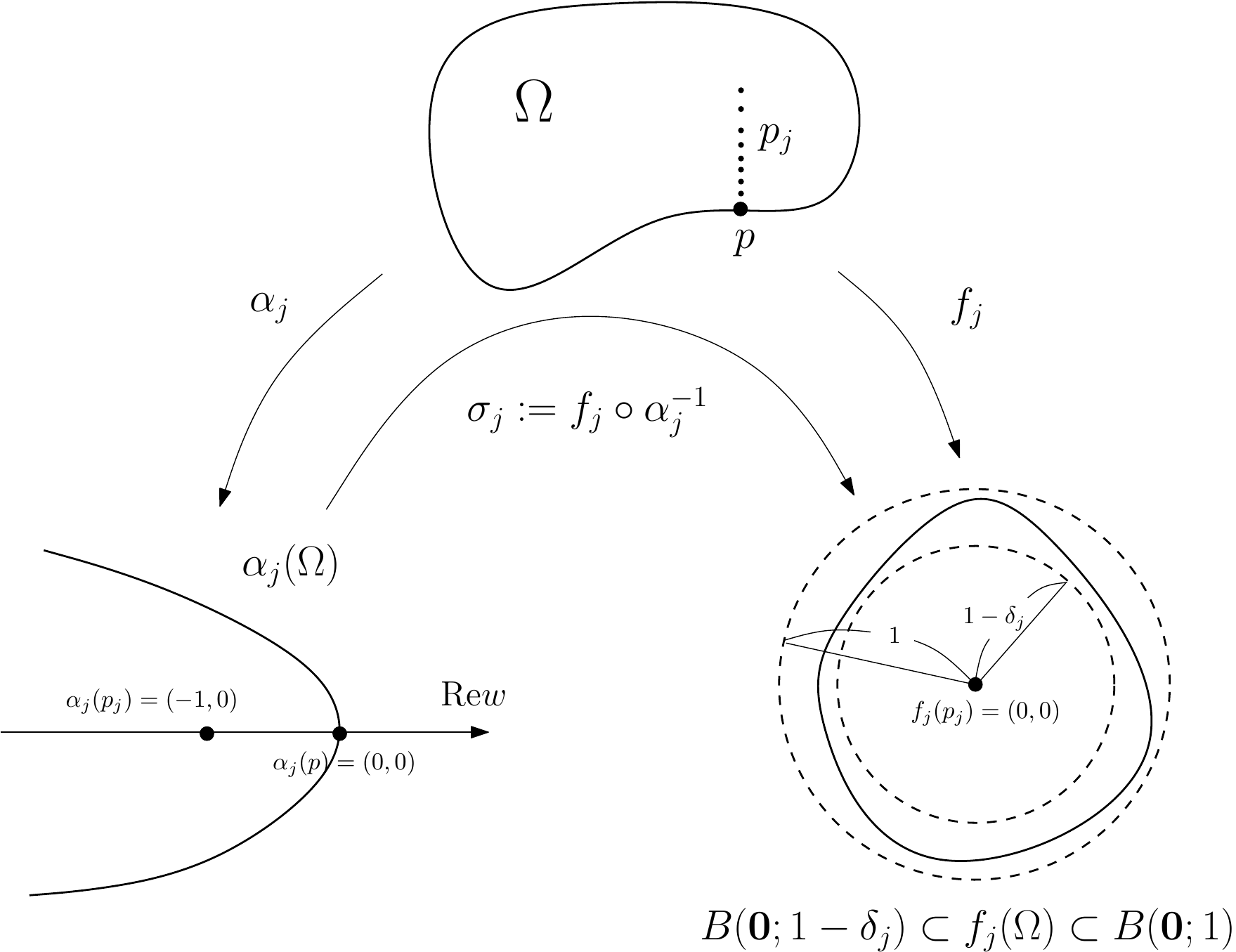}
\end{figure}
\bigskip

Define $\sigma_j : \alpha_j(\Omega) \to f_j(\Omega)$ by
$\sigma_j := f_j \circ \alpha_j^{-1}$. Note that $\sigma_j(-1, 0) = (0, 0)$ 
for every $j$.

\section{Convergence of the ``reverse'' scaling sequence $\{ \sigma_j \}$}

Recall the concept of normal set-convergence introduced in 
\cite{GKK}, Section 9.2.2; it will give the necessary control for the 
convergence. 

\begin{defn} \rm
Let $\Omega_j$ be domains in $\mathbb{C}^n$ for each $j = 1, 2, \cdots$. The sequence $\Omega_j$ is said to {\it converge normally} to a domain $\widehat{\Omega}$, if the following two conditions hold:
\begin{enumerate}
\item For any compact set $K$ contained in the interior of $\bigcap_{j>m} \Omega_j$ for some positive integer $m$, $K \subset \widehat{\Omega}$.
\item For any compact subset $K'$ of $\widehat{\Omega}$, there exists a constant $m > 0$ such that $K' \subset \bigcap_{j>m} \Omega_j$.
\end{enumerate}
\end{defn}

\begin{prop}
If $\Omega_j$ is a sequence of domains in $\mathbb{C}^n$ that 
converges normally to the domain $\widehat{\Omega}$, then
\begin{enumerate}
\item If a sequence of holomorphic mappings 
$f_j : \Omega_j \rightarrow \Omega '$ from $\Omega_j$ to another 
domain $\Omega '$ converges uniformly on compact subsets of 
$\widehat{\Omega}$, then its limit is a holomorphic mapping from 
$\widehat{\Omega}$ into the closure of the domain $\Omega '$.
\item If a sequence of holomorphic mappings 
$g_j : \Omega ' \rightarrow \Omega_j$ converges uniformly on 
compact subsets of $\Omega '$, if $\widehat{\Omega}$ is 
pseudoconvex, and if there are a point $p \in \Omega '$ and a 
constant $c > 0$ so that the inequality $|\det{(dg_j|_p)}| > c$ holds 
for each $j$, then $\lim_{j \to \infty}g_j$ is a holomorphic mapping 
from the domain $\Omega '$ into $\widehat{\Omega}$.
\end{enumerate}
\end{prop}

In our construction, the set-convergences 
$f_j(\Omega) \to \mathbb{B}^2$ and 
$\alpha_j(\Omega) \to \widehat{\Omega}$ are in accordance with the 
sense of normal set-convergence with 
\[
\widehat{\Omega} := \{ (w, z) \in \mathbb{C}^2 \mid \widehat{\rho} 
= \textrm{Re}\,w + \widehat{P} (z, \bar{z}) < 0 \}.
\] 
Notice that $\widehat{\Omega}$ is unbounded, so the convergence of 
the \textit{forward scaling sequence} $\{ \sigma_j^{-1} \}$ is not 
immediately obvious. On the other hand, one easily observes that
the inverse sequence (i.e., \textit{the reverse scaling sequence}) 
$\{ \sigma_j \}$ converges, choosing a subsequence when necessary, 
by Proposition 3.2 and Montel's theorem.  So we take a convergent
subsequence and call it $\{ \sigma_j \}$ again, and denote the limit map 
by $\widehat{\sigma}$. 
Since $\widehat{\sigma}(-1, 0) = (0, 0) \in \mathbb{B}^2$, it holds that
$\widehat{\sigma} (\widehat\Omega) \subset \mathbb{B}^2$. 
\smallskip

Now we show:

\begin{prop}
$\widehat{\sigma} : \widehat{\Omega} \rightarrow \mathbb{B}^2$ is 
a biholomorphic map.
\end{prop}

\begin{proof}
The proof is almost the same as those for Propositions 2.8 and 2.10
of \cite{JooSR}.  However, the surjectivity part of $\widehat{\sigma}$ 
requires a few, simple but perhaps subtle adjustments. Therefore, we 
choose to include the detail here. 
\smallskip

Suppose that $\widehat{\sigma}$ is not onto. Then there is  
a boundary point $q$ of $\widehat{\sigma}(\widehat{\Omega})$ 
in $\mathbb{B}^2$. Notice that $f_j^{-1}(q)$ converges 
to $p$ since $p$ is a peak point of $\partial\Omega$. 
Denote by $q_j := f_j^{-1}(q)$. Now we construct a new scaling sequence 
$\sigma_j^q := f_j \circ (\alpha_j^q)^{-1}$ where 
$\alpha_j^q$ is a stretching map (in the sense of Pinchuk) with respect 
to $q_j$ as above. 
In the same way, there 
is a subsequential limit map $\widehat{\sigma^q} : \widehat{\Omega^q} 
\to \mathbb{B}^2$ of $\{ \sigma_j^q \}$, where $\widehat{\Omega^q}$ is 
the limit domain of the sequence $\alpha_j^q(\Omega)$ in the sense of normal set-convergence. 
We have already observed that $\widehat{\sigma^q}$ is 1-1. Taking 
a subsequence if necessary, we may assume that the uniform 
convergence holds for $\sigma_j^{-1} \to \widehat{\sigma}^{-1}$ and 
$(\sigma_j^q)^{-1} \to \widehat{\sigma^q}^{-1}$ on compact subsets of 
$\widehat{\sigma}(\widehat{\Omega})$ and $\widehat{\sigma^q}(\widehat{\Omega^q})$ respectively. 
(See Lemma 2.9 in \cite{JooSR}).
\smallskip

Denote by $W := \widehat{\sigma}(\widehat{\Omega}) \cap 
\widehat{\sigma^q}(\widehat{\Omega^q})$. Then the map 
$\beta_j := (\sigma_j^q)^{-1} \circ \sigma_j : \sigma_j^{-1}(W) \to 
(\sigma_j^q)^{-1}(W)$ is well-defined.
Actually, $\beta_j \equiv \alpha_j^q \circ \alpha_j^{-1}$, and this, for each 
$j$, is a polynomial automorphism of $\mathbb{C}^2$ with degree less 
than or equal to $2k$. On the other hand, $\beta_j$ converges to 
$\widehat{\beta} := (\widehat{\sigma^q})^{-1} \circ 
\widehat{\sigma}$ uniformly on compact subsets of 
$\widehat{\sigma}^{-1}(W)$. By calculation, $\widehat{\beta}$ turns out to 
be a polynomial automorphism of $\mathbb{C}^2$ of degree less than or 
equal to $2k$. Now Proposition 3.2 guarantees that the restriction 
$\widehat{\beta}|_{\widehat{\Omega}}$ is a 1-1 holomorphic map from 
$\widehat{\Omega}$ into $\widehat{\Omega^q}$. 
In a similar way for the inverse sequence $\{ \beta_j^{-1} \}$, the map 
$(\widehat{\beta})^{-1}|_{\widehat{\Omega^q}}$ is a 1-1 holomorphic map 
from $\widehat{\Omega^q}$ into $\widehat{\Omega}$. Hence 
$\widehat \beta \colon \widehat{\Omega} \to \widehat{\Omega^q}$ is 
a biholomorphism.

\begin{figure}[!h]
\centering
\includegraphics[height=2.4in, width=3.3in, angle=0]{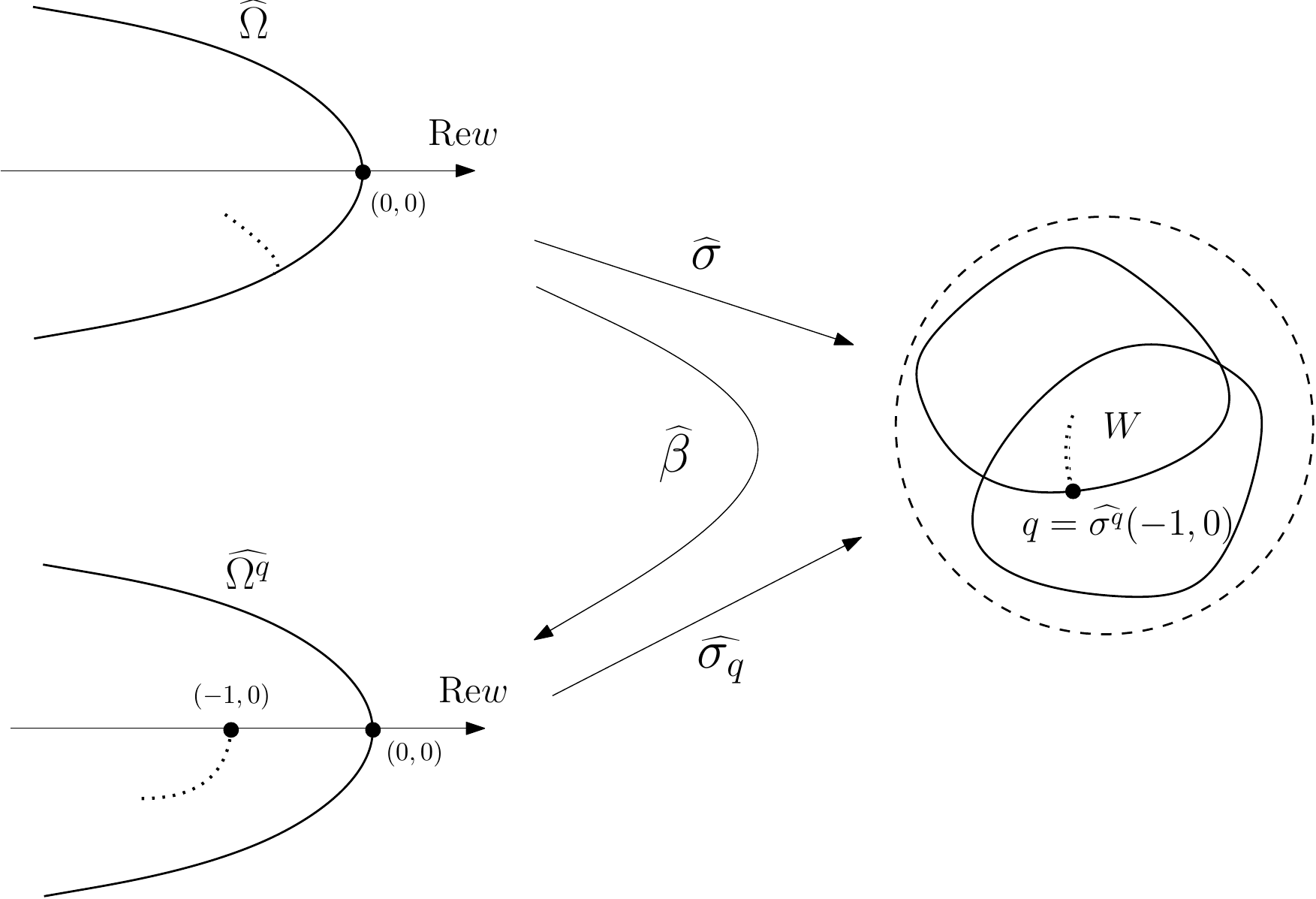}
\end{figure}
\bigskip

On the other hand, we see that a sequence of points in \(W\)
convergent to $q$ gives rise to a sequence in \(\widehat\Omega\) via
\(\widehat\sigma\) that approaches the boundary 
\(\partial \widehat\Omega\)
and also to a sequence in \(\widehat\Omega^q\) via
\(\widehat\sigma^q\) which, this time, converges to the interior point
\((-1,0)\).  This results in that the biholomorphism \(\widehat\beta
\colon \widehat\Omega \to \widehat\Omega^q\) maps a sequence 
approaching the boundary to a sequence convergent to an interior point.
So \(\widehat\beta\) fails to be proper, and the surjectivity of 
$\widehat{\sigma}$ follows by this contradiction.
\end{proof}

\section{Proof of the Theorem \ref{main}} 

We are ready to complete the proof of Theorem \ref{main}.
Recall that $\widehat{\Omega} := \{ (w, z) \in \mathbb{C}^2 \mid 
\textrm{Re}\,w + \widehat{P} (z, \bar{z}) < 0 \}$, where $\widehat{P}$ is a 
nonzero real-valued subharmonic polynomial of degree $2k$. Note that the 
unit ball $\mathbb{B}^2$ is biholomorphic to the Siegel half space 
$\{ (w, z) \in \mathbb{C}^2 \mid \textrm{Re}\,w + |z|^2 < 0 \}$. So the 
theorem of Oeljeklaus in \cite{Oelj}, which says that these two domains 
must be affinely biholomorphic in such a case, implies in particular that 
$\widehat{P} (z, \bar{z}) = c|z|^2$ for some $c > 0$. Therefore, the origin 
must have been the boundary point of type 2 in the first place.  Since 
the convergence $\rho_j \to \textrm{Re}\,w + c|z|^2$ is uniform on 
each jet-level, there exist positive contants $c$ and $j_0$ such that 
the smallest eigenvalue of the Levi form of each $\rho_j$ at $(0, 0)$ is 
larger than $c$ whenever \(j > j_0\).  Consequently, the Levi form of 
$\rho$ at $q$ is strictly positive-definite. This completes the proof of 
Theorem \ref{main}.
\hfill \(\Box\)

\section{A remark for the convex case}

Notice that the scaling sequence converges without difficulties in case
the domain is bounded convex. See e.g., \cite{KK}.  If the 
boundary point \(p\) under consideration satisfies the condition
\( \lim_{\Omega \ni q \to p} s_\Omega (q)=1\) and if \(p\) were 
\(C^\infty\) convex of 
infinite type, then the scaled limit turns out to be a convex domain 
that contains a one-dimensional disc with a positive radius in its 
boundary. (cf.\ \cite{Zimm}, Theorem 3.7). 

\begin{figure}[!h]
\centering
\includegraphics[height=2in, width=2.5in, angle=0]{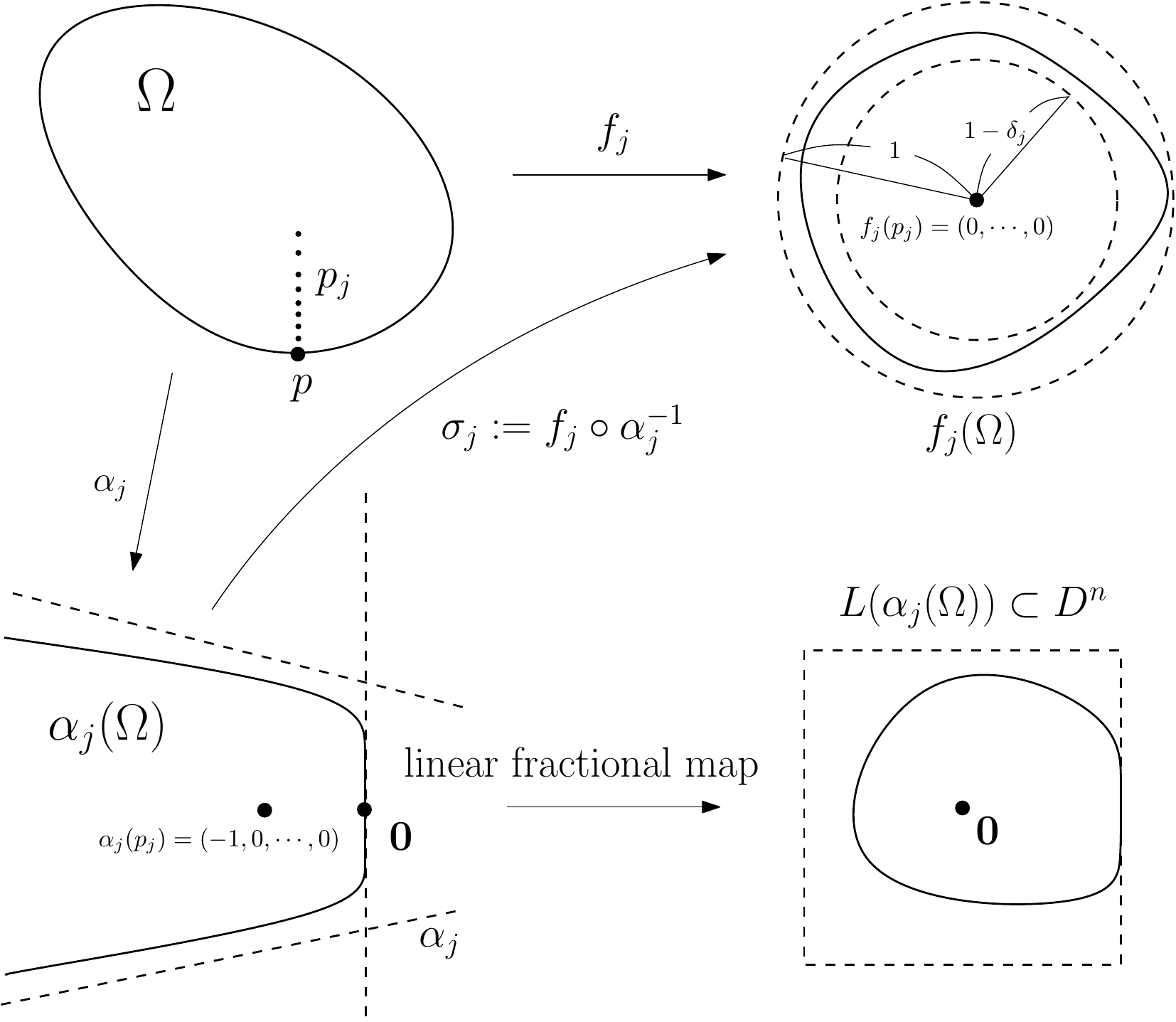}
\end{figure}

On the other hand, the preceding arguments imply that this domain has 
to be biholomorphic to the ball.  But this is impossible, for instance by
\cite{Huc}.  So the boundary point \(p\) has to be of finite type. Notice
that our arguments work in all dimensions in this case, as the 
scaling method for bounded convex domains converges regardless of
dimension (cf., e.g., \cite{KK}).  This reconfirms Zimmer's 
affirmative answer \cite{Zimm} to Question \ref{Fornaess} for the 
smoothly bounded convex domains in \(\mathbb C^n\) for all \(n\).


\begin{thebibliography}{99}

\bibitem{BP} E. Bedford and S. I. Pinchuk: Domains in 
${\mathbb C}^{n+1}$ with noncompact automorphism group. J. Geom. Anal. 1 (1991), 
no. 3, 165--191.

\bibitem{CY} S.-Y. Cheng and S.-T. Yau: Complete affine hypersurfaces. I. 
The completeness of affine metrics. Comm. Pure Appl. Math. 39 (1986), no. 
6, 839--866.

\bibitem{DFW} K. Diederich, J. E. Forn{\ae}ss, E. F. Wold: Exposing points 
on the boundary of a strictly pseudoconvex or a locally convexifiable 
domain of finite 1-type. {\it J. Geom. Anal.} {\bf 24} (2014), 2124--2134.

\bibitem{DAN} J. P. D'Angelo: Real hypersurfaces, orders of contact, and 
applications. Ann. of Math. (2) 115 (1982), no. 3, 615--637.

\bibitem{DGZ} F. Deng, Q. Guan, L. Zhang: Properties of squeezing 
functions and global transformations of bounded domains. {\it Trans. 
Amer. Math. Soc} {\bf 368} (2016), 2679--2696.

\bibitem{FR} J. E. Forn{\ae}ss, F. Rong: Estimates of the squeezing 
function for a class of bounded domains, \textit{Arxiv:1606.01335} (2016).

\bibitem{FW} J.E.  Forn{\ae}ss, E. F. Wold: A non-strictly pseudoconvex 
domain for which the squeezing function tends to one towards the 
boundary, \textit{Arxiv:1611.04464} (2016).

\bibitem{GKK} R. E. Greene, K.-T. Kim, S. G. Krantz: 
{\it The geometry of complex domains}, Progress in Math. {\bf 291}, 
Birkh\"auser, 2010.

\bibitem{Huc} A. Huckleberry: Holomorphic fibrations of bounded domains. 
Math. Ann. 227 (1977), no. 1, 61–66.

\bibitem{JooSR} S.-R. Joo: On the scaling methods by Pinchuk and 
Frankel. {\it arXiv:1607.06580} (2016).

\bibitem{KK} K.-T. Kim and S. G. Krantz: Complex scaling and domains with 
non-compact automorphism group. Illinois J. Math. 45 (2001), no. 4, 
1273--1299.

\bibitem{KZ} K.-T. Kim, L. Zhang: On the uniform squeezing property 
and the squeezing function. {\it Pacific J. Math.} {\bf 282} (2016), no. 2, 
341--358.

\bibitem{LSY1} K. Liu, X. Sun, S.-T. Yau: Canonical metrics on the moduli 
space of Riemann surfaces. I. J. Differential Geom. 68 (2004), 
no. 3, 571--637.

\bibitem{LSY2} K. Liu, X. Sun, S.-T. Yau: Canonical metrics on the moduli 
space of Riemann surfaces. II. J. Differential Geom. 69 (2005), no. 1, 
163--216.

\bibitem{Oelj} K. Oeljeklaus: On the automorphism group of 
certain hyperbolic domains in $\mathbb{C}^2$. 
{\it Asterisque} {\bf 217} (1993), 193-216.

\bibitem{Pin} S. I. Pinchuk: Holomorphic inequivalence of certain classes 
of domains in $\mathbb C^n$. (Russian) Mat. Sb. (N.S.) 111(153) (1980), 
no. 1, 67--94.

\bibitem{Yeung} S. K. Yeung: Geometry of domains with the uniform 
squeezing property. Adv. Math. 221 (2009), no. 2, 547–569.

\bibitem{Zimm} A. Zimmer: A gap theorem for the complex geometry 
of convex domains. {\it arXiv:1609.07050} (2016).
\end{thebibliography}
\end{document}